%% file: compositionalityarticle-doc.tex
\documentclass[a4paper,onecolumn,superscriptaddress,10pt,unpublished]{compositionalityarticle}
\pdfoutput=1

\usepackage[utf8]{inputenc}
\usepackage[english]{babel}
\usepackage[T1]{fontenc}
\usepackage{amsmath}
\usepackage{hyperref}
\usepackage{tikz}
\usepackage[numbers,sort&compress]{natbib}

\input{misc/preamble}

\newcommand{\authorsforheader}{Pugh, Grundy, Cirstea and Harris}
\newcommand{\compositionalityIssue}{53}
\newcommand{\compositionalityVolume}{0}
\newcommand{\paperdoinumber}{10.46298/compositionality-\compositionalityVolume-\compositionalityIssue}
\newcommand{\paperdoi}{https://doi.org/\paperdoinumber}
\newcommand{\receiveddate}{yyyy-mm-dd}
\newcommand{\accepteddate}{yyyy-mm-dd}

\begin{document}

\title{Presenting Neural Networks via Coherent Functors}
\date{}
\author{Matthew Pugh}
\email{mtp47@cam.ac.uk}
\orcid{0009-0003-5410-5529}
\affiliation{University of Cambridge: Cambridge, England, GB }
\author{Jo Grundy}
\email{j.grundy@soton.ac.uk}
\orcid{0000-0003-2583-5680}
\affiliation{University of Southampton: Southampton, GB }
\author{Corina Cirstea}
\email{cc2@ecs.soton.ac.uk}
\orcid{0000-0003-3165-5678}
\affiliation{University of Southampton: Southampton, GB }
\author{Nick Harris}
\email{nrh@ecs.soton.ac.uk}
\orcid{0000-0003-4122-2219}
\affiliation{University of Southampton: Southampton, GB }
\maketitle

\input{sections/abstract}
\input{sections/introduction}
\input{sections/background}
\input{sections/datasets}
\input{sections/ml_models}
\input{sections/conclusion}

\bibliographystyle{plainnat}
\bibliography{misc/references.bib}

\end{document}

%% file: misc/preamble.tex
\usepackage{quiver}
\usepackage{amsmath}
\usepackage{amsfonts}
\usepackage{amssymb}
\usepackage{mathtools}
\usepackage{amsthm}
\usepackage{enumitem}
\usepackage[capitalize,noabbrev]{cleveref}

\theoremstyle{plain}
\newtheorem{theorem}{Theorem}[section]

\newtheorem{lemma}[theorem]{Lemma}
\newtheorem{corollary}[theorem]{Corollary}
\theoremstyle{definition}
\newtheorem{definition}[theorem]{Definition}

\theoremstyle{remark}

%% file: sections/abstract.tex
\begin{abstract}
This paper develops a methodology for representing machine learning models as models of formal theories, grounded in the perspective that machine learning models are a form of database and that databases are models of theories in coherent logic. Two intermediate results support this approach: any functorial database schema has an associated $\kappa$-coherent theory whose models coincide with its instances, and data may be hard-coded into a coherent category such that any model of the resulting theory necessarily contains it. These tools are used to show that any dense feed-forward neural network architecture over the floating point numbers may be presented as a coherent category $G$ whose $Set$-models are the networks of that architecture, with inference arising as the precomposition functor $Coh(\iota, Set)$ along a coherent functor $\iota : RSpan(a_0, a_n) \rightarrow G$. This representation is extended to networks with weight and bias fixing and tying, encompassing sparse and convolutional architectures, via a 2-coequaliser construction in $Coh_\sim$. Taken together, these results recast neural network inference as an extension problem in the 2-category $Coh_\sim$ of coherent categories, supporting the interpretation of a network architecture as a formal hypothesis about the structure of data and of model training as a lifting of a dataset into a more constrained theory.
\end{abstract}

%% file: sections/introduction.tex
\section{Introduction}

Recent work has demonstrated that any error minimisation problem may be represented as an extension problem within a suitable 2-category such that the left Kan extensions are, if they exist, the solutions to the error minimisation problem \cite{pugh_learning_2025}. The question of how one should represent machine learning models within category theory is an active area of research \cite{shiebler_category_2021, gavranovic_position_2024}. A clearer connection between the description of a machine learning problem and its representation as an extension in a 2-category would provide a new avenue for the investigation of machine learning algorithms.

This paper proposes a methodology for the representation of machine-learning algorithms as models of formal theories and demonstrates this approach by representing all feed-forward neural networks and their associated model inference as an extension problem in the category of coherent categories (Thm \ref{theorem:neural_network_representation}). While presenting concrete results at a low level, this work also supports an alternative higher level perspective of machine learning, whereby formal theories characterise the suitable datasets and machine learning models. The representation of machine learning inference via the precomposition of a coherent functor suggests that any machine learning model contains an internal representation of a dataset which is extracted through inference. It also represents the process of learning itself as a lifting or extensions of a dataset into a formal theory which adds additional constraints and assumptions. That is, a machine learning architecture is a hypothesis about the structure of a dataset and the model returned via training is the best representation of the data under these additional constraints. The 2-morphisms of $Coh_\sim$, as natural isomorphisms, are not immediately suited to realising the Kan extension as a loss-minimising model in the sense of \cite{pugh_learning_2025}; whether this may be achieved by appending suitable 2-morphisms or restricting to an appropriate sub-2-category is identified as a direction for future work.

It is shown that a simple description of databases as functors \cite{spivak_functorial_2013} can be lifted into a representation using $\kappa$-coherent categories (Thm \ref{theorem:schema}), indicating that the representation of databases as functors on small categories is equivalent to models of theories of some fragment of infinitary coherent logic. It is then shown that known information encoded in some simple database structure may be hard-coded into a theory such that any model of this theory contains the original database (Thm \ref{theorem:hard_coding}).

The hard coding theorem is included as a shortcut for constructing theories of machine learning models. It allows necessary data to be hard-coded into a coherent theory where subsequent model constraints may be introduced through 2-colimits of coherent categories. This approach is used to show that dense neural networks and networks with fixed and tied weights (such as convolutional architectures) may be presented as an extension along coherent functors (Thm \ref{theorem:neural_network_representation} and Thm \ref{corollary:weight_tying}).

The paper is structured as follows. Section \ref{section:background} introduces the relevant background on $\kappa$-coherent logic, syntactic categories, and the relationship between coherent theories and their models. Section \ref{section:datasets} demonstrates how functorial database schemas lift to coherent theories and develops the Hard Coding theorem. Section \ref{section:mlmodels} applies these tools to construct coherent theories of dense neural network architectures and their inference functors.

%% file: sections/background.tex
\section{Background}
\label{section:background}

It is assumed that the reader will have a basic knowledge of category theory but possibly not of $\kappa$-coherent theories and their respective categories. The following section will provide a brief introduction of the relevant definitions and theorems as can be found in the following references \cite{makkai_first_1977, kanalas_21-category_2021, mac_lane_sheaves_1994}.

The family of $\kappa$-coherent logics are distinguished by what sizes of disjunctions (unions) they allow. $\kappa$ is a regular cardinal such that for any set of predicates where the size of the set is less than $\kappa$, then the disjunction of this collection of predicates must exist in a $\kappa$-coherent logic. Consequently, many of the deductions of the following section are sensitive to the relative sizes of the sets involved. Definitions for the relevant terms may be found as follows.

\begin{definition}[$Set_{<\kappa}$]
Given a cardinal number $\kappa$ (An isomorphism class of sets) then $Set_{<\kappa}$ is the category of sets with cardinality strictly less than $\kappa$.
\end{definition}

\begin{definition}[$Set_{\leq\kappa}$]
Given a cardinal number $\kappa$ (An isomorphism class of sets) then $Set_{\leq\kappa}$ is the category of sets with cardinality less than or equal to $\kappa$.
\end{definition}

\begin{definition}[$\kappa$ regular cardinal]
A regular cardinal is a cardinal number $\kappa$ such that $Set_{<\kappa}$ has all colimits of diagrams whose indexing category has a set of morphisms with cardinality strictly less than $\kappa$.
\end{definition}

\begin{definition}[$\kappa$-small set]
A set is $\kappa$-small if it has cardinality strictly less than $\kappa$.
\end{definition}

The paper's overarching concept is to use the relationship between theories, categories, and databases to make intuitions about the behaviour of machine learning models more precise. Beyond what is presented in this work, there exists a relationship between type theories in general and their syntactic categories \cite{pitts_categorical_2001, schreiber_syntactic_2015} which indicates that this perspective has a powerful generality. However, for the purposes of this paper, it is sufficient to focus only on the definitions and theorems relevant to the $\kappa$-coherent category. As such, the following definitions will be presented with a focus on this subset of logic.

A $\kappa$-coherent logic makes use of distinguished symbols which represent sorts, predicates, and operations, which may be combined through the use of logic connectives ($\vee$, $\wedge$, $\bot$, $\top$, $\exists$) to produce $\kappa$-coherent formulae.

\begin{definition}[Signature]
A signature $L$ has three classes of symbols $\langle S, P, F \rangle$. $S$ is a set of sorts $s \in S$ (also called types). $P$ is a set of finitary sorted predicate symbols $p \in P$ sometimes written as $p \subseteq s_1 \times ... \times s_n$. $F$ is a set of finitary sorted operation symbols $f \in F$ often written as $f : s_1\times ... \times s_n \rightarrow s_o$.
\end{definition}

\begin{definition}[Free variable]
A free variable $x$ of sort $s$ may be written as $x:s$ and indicates where, in an expression, a term of the same sort may be substituted.
\end{definition}

\begin{definition}[Terms of sort $s$]
A term $t$ of sort $s$ may be written as $t:s$. A term may be a free variable, a constant, or may be formed from an operation symbol $f: s_1\times ... \times s_n \rightarrow s_o$ and terms $t_i:s_i$ written as $f(t_1,...,t_n):s_o$.
\end{definition}

\begin{definition}[Atomic Formulae]
Given a signature $L$, if $p\subseteq s_1 \times ... \times s_n$ is a predicate symbol then $p(t_1, ... , t_n$) is an atomic formula for terms $t_i:s_i$. If $t$ and $t'$ are terms of the same sort, then $t = t'$ is an atomic formula. The symbols $\top$ and $\bot$ are atomic formulas corresponding to the identically true and false formulas.
\end{definition}

\begin{definition}[$\kappa$-Coherent Formulae]
Given a signature $L$, The class $\chi_\kappa$ of $\kappa$-coherent formulae over $L$ contains:
\begin{itemize}[noitemsep,topsep=-8pt]
   \item All atomic formulae of $L$.
   \item The conjunction $\bigwedge\Theta$ for $\Theta \subseteq \chi_\kappa$ where $\Theta$ is finite.
   \item The disjunction $\bigvee\Theta$ for $\Theta \subseteq \chi_\kappa$ where $\Theta$ is $\kappa$-small.
\end{itemize}
\end{definition}

\begin{definition}[Infinitary-Coherent Formula]
An infinitary-coherent formula is a $\kappa$-coherent formula for some choice of $\kappa$.
\end{definition}

The collection of $\kappa$-coherent formulae represents the statements that can be made within a $\kappa$-coherent logic. A theory within this language dictates the implication relation between subsets of formulae.

\begin{definition}[$\kappa$-Coherent Context]
A $\kappa$-coherent context $\Phi$ is a finite subset of $\chi_\kappa$.
\end{definition}

\begin{definition}[$\kappa$-Coherent Sequent]
A $\kappa$-coherent sequent is an ordered pair of $\kappa$-coherent contexts ($\Phi$ and $\Psi$) written as $\Phi \vdash \Psi$. It corresponds to the statement that for any substitution of variables, if $\Phi$ is true then $\Psi$ is true.
\end{definition}

\begin{definition}[$\kappa$-Coherent Theory]
A $\kappa$-coherent theory is a set of $\kappa$-coherent sequents
\end{definition}

A model of a theory in a category represents the coherent formulae used by a given theory as diagrams and (co)limits of the category such that any sequent of the theory is valid within the interpretation in the category. This ensures that the model always fulfils the constraints imposed by the theory.

\begin{definition}[$C$-structure of type $L$]
Given a signature $L$ and a category $C$ which has all finite limits, then a $C$-structure of type $L$ is a mapping $M$, that sends sorts $s$ of $L$ to objects $M(s)$ of $C$, sends predicate symbols $p\subseteq s_1 \times... \times s_n$ to sub-objects (monomorphisms) $M(p): X \rightarrowtail M(s_1) \times ... \times M(s_n)$ of $C$, and sends operation symbols $f : s_1 \times ... \times s_n \rightarrow s_o$ to morphisms $M(f) : M(s_1)\times ... \times M(s_n) \rightarrow M(s_o)$ of $C$. 
\end{definition}

\begin{definition}[Interpreting formulas of $C$-structures]
\label{def:interpretation}
Given a $C$-structure $M$, a formula of $M$ may be interpreted as follows.
\begin{itemize}[noitemsep,topsep=-8pt]
   \item A finite sequence of free variables $\vec{x} = (x_1,...,x_n)$ of sorts $x_i:s_i$ is interpreted as $M(x_1,...,x_n) = M(\vec{x})=M(s_1)\times...\times M(s_n)$.
   \item A term $t(\vec{x}):s$ may be interpreted as a morphism $M_{\vec{x}}(t) : M(\vec{x}) \rightarrow M(s)$ such that: If $t=x_i$ then $M_{\vec{x}}(t)$ is the corresponding projection morphism of the product construction. If $t = f(t_1,...,t_n)$, then $M_{\vec{x}}(t)$ is the composite $M(f)\langle M_{\vec{x}}(t_1),..., M_{\vec{x}}(t_n) \rangle$.
   \item A formula over free variables $\phi(\vec{x})$ is interpreted as a subobject $M_{\vec{x}}(\phi) : X \rightarrowtail M(\vec{x})$
   \item The atomic formula $t_1 = t_2 : s$ is interpreted as the equaliser (limit of the diagram of parallel morphisms) of $M_{\vec{x}}(t_1)$ and $M_{\vec{x}}(t_2)$.
\[\begin{tikzcd}
	X && {M(\vec{x})} && {M(s)}
	\arrow["{M_{\vec{x}}(t_1=t_2)}", tail, from=1-1, to=1-3]
	\arrow["{M_{\vec{x}}(t_1)}", shift left=2, from=1-3, to=1-5]
	\arrow["{M_{\vec{x}}(t_2)}"', shift right=2, from=1-3, to=1-5]
\end{tikzcd}\]
\item A predicate with substituted terms $p(\vec{t})$ for $t_i : s_i$ is given by the following pullback diagram.
\[\begin{tikzcd}
	{X'} & {M(\vec{x})} \\
	X & {M(s_1)\times...\times M(s_n)}
	\arrow["{M(p(\vec{t}))}", tail, from=1-1, to=1-2]
	\arrow[from=1-1, to=2-1]
	\arrow["\ulcorner"{anchor=center, pos=0.125}, draw=none, from=1-1, to=2-2]
	\arrow["{\langle M_{\vec{x}}(t_1),...,M_{\vec{x}}(t_n)\rangle}", from=1-2, to=2-2]
	\arrow["{M(p)}"', tail, from=2-1, to=2-2]
\end{tikzcd}\]
\item The conjunction of formulae is the infimum of their corresponding sub-objects, .i.e a pullback along the monomorphisms $M_{\vec{x}}(\bigwedge \Theta)=\bigwedge\{M_{\vec{x}}(\phi)\ |\ \phi \in \Theta\}$
\item The disjunction of formulae is the supremum of their corresponding sub-objects, i.e. a co-product in the category of sub-objects, or a push-out along the pullback of the monomorphisms  $M_{\vec{x}}(\bigvee \Theta)=\bigvee\{M_{\vec{x}}(\phi)\ |\ \phi \in \Theta\}$
\item The existential quantifier of a formula $\exists y.\varphi$ is the following image factorisation.
\[\begin{tikzcd}
	X & {M(\vec{x},y)} & {M(\vec{x})} \\
	& {Im(\pi_{\vec{x}}M(\phi))}
	\arrow["{M(\phi)}", tail, from=1-1, to=1-2]
	\arrow[two heads, from=1-1, to=2-2]
	\arrow["{\pi_{\vec{x}}}", from=1-2, to=1-3]
	\arrow["{M_{\vec{x}}(\exists y.\phi)}"', tail, from=2-2, to=1-3]
\end{tikzcd}\]
\end{itemize}
\end{definition}

\begin{definition}[Valid]
\label{def:valid}
The sequent $\Phi \vdash \Psi$ is valid in the structure $M$, written as $M\vDash \Phi \vdash \Psi$, if and only if $\bigwedge\{M_{\vec{x}}(\phi)\ |\ \phi \in \Phi\} \leq \bigvee\{M_{\vec{x}}(\psi)\ |\ \psi \in \Psi\}$\footnote{The inequality is in reference to the typical subobject preorder structure.}
\end{definition}

\begin{definition}[$T$-model]
\label{def:model}
A $C$-structure $M$ of type $L$ is a model of the theory $T$, a $T$-model, if $L$ is the signature of $T$ and all sequents of $T$ are valid in $M$.
\end{definition}

\begin{definition}[$T$-model morphism]
\label{def:model_morphisms}
Given two $T$-models $M$ and $M'$ in a category $C$, a morphism of models $\alpha : M \rightarrow M'$ consists of a family of arrow $\alpha_s : M(s) \rightarrow M'(s)$ for every sort $s$ in the language of $T$, such that for every operation symbol $f: s_1 \times ... \times s_n \rightarrow s_o$ the following diagram commutes.
\[\begin{tikzcd}
	{M(s_1)\times...\times M(s_n)} & {M(s_o)} \\
	{M'(s_1)\times ... \times M'(s_n)} & {M'(s_o)}
	\arrow["{M(f)}", from=1-1, to=1-2]
	\arrow["{\alpha_{s_1}\times...\times\alpha_{s_n}}"', from=1-1, to=2-1]
	\arrow["{\alpha_{s_o}}", from=1-2, to=2-2]
	\arrow["{M'(f)}"', from=2-1, to=2-2]
\end{tikzcd}\]
And for every predicate symbol $p \subseteq s_1 \times ... \times s_n$ the following diagram commutes and the dashed arrow in the diagram exists.
\[\begin{tikzcd}
	X & {M(s_1)\times...\times M(s_n)} \\
	Y \\
	{X'} & {M(s_1)\times ... \times M(s_n)}
	\arrow["{M(p)}", tail, from=1-1, to=1-2]
	\arrow[two heads, from=1-1, to=2-1]
	\arrow["{\alpha_{s_1}\times...\times\alpha_{s_n}}", from=1-2, to=3-2]
	\arrow[dashed, from=2-1, to=3-1]
	\arrow[tail, from=2-1, to=3-2]
	\arrow["{M'(p)}"', tail, from=3-1, to=3-2]
\end{tikzcd}\]
\end{definition}

\begin{definition}[$T\text{-}Mod(C)$]
The category $T\text{-}Mod(C)$ has $T$ models in $C$ as its objects and $T$-model morphisms as its morphisms.
\end{definition}

As can be seen from Def \ref{def:interpretation}, a category needs to possess certain internal structures in order for it to be possible to interpret a $\kappa$-coherent theory inside of it. The categories with such a structure are $\kappa$-coherent categories and the functors which preserve these structures are $\kappa$-coherent functors.

\begin{definition}[$\kappa$-Coherent Category
\footnote{These categories are referred to by different names in the literature $\kappa$-coherent, $\kappa$-geometric, or as in \cite{makkai_first_1977} $\kappa$-logical. They are referred to here as $\kappa$-coherent as that appears to be the most consistent terminology.} \cite{makkai_first_1977, schreiber_geometric_2022}]
Given a regular cardinal $\kappa$, a $\kappa$-coherent category has all finite limits, pullback-stable image factorisations, and pullback-stable unions of $\kappa$-small families of sub-objects.
\end{definition}

The traditional notion of a coherent category has only finite unions, so would be an $\aleph_0$-coherent category. An infinitary coherent category or geometric category has all unions of sub-objects.

\begin{definition}[$\kappa$-Coherent Functor]
A $\kappa$-coherent functor preserves all finite limits, image factorisations and $\kappa$-small unions.
\end{definition}

$\kappa$-coherent categories and functors form a subcategory of $Cat$ referred to as $\kappa\text{-}Coh$. $\kappa\text{-}Coh$ is, in fact, a 2-category, possessing all natural transformations between functors it contains. Of particular relevance to this work is the 2-category of coherent categories ($\aleph_0$-coherent categories), which only possesses natural isomorphism, referred to as $Coh_\sim$. This is particularly important because Section \ref{section:mlmodels} requires the use of 2-(co)limits in $Coh_\sim$ as part of the construction of theories of neural networks.

\begin{definition}[$\kappa\text{-}Coh$]
$\kappa$-Coh is a 2-subcategory of $Cat$ which contains all $\kappa$-coherent categories and $\kappa$-coherent functors and all natural transforms.
\end{definition}

\begin{definition}[$\kappa\text{-}Coh_\sim$]
$\kappa\text{-}Coh_\sim$ is a (2,1)-subcategory of $Cat$ which contains all $\kappa$-coherent categories and $\kappa$-coherent functors and all natural isomorphisms.
\end{definition}

\begin{theorem}[Theorem 4.12 in \cite{kanalas_21-category_2021}]
\label{theorem:bicomplete}
$Coh_\sim$ is 2-complete and 2-cocomplete.
\end{theorem}

Not only do $\kappa$-coherent categories provide the suitable structure for interpreting $\kappa$-coherent theories, but each one is associated with a $\kappa$-coherent theory. This comes from constructing a signature using the objects and morphisms of the $\kappa$-coherent category. Then, sequents are added to reproduce the relevant commutative diagrams of the category.

\begin{definition}[Canonical Signature\footnote{This is also referred to as the canonical language of the category as in some portions of the literature a signature is referred to as a language}]
\label{definition:canonical_signature}
The canonical signature of a category $C$ is the signature $L_C = \langle Ob(C), Mor(C), \varnothing \rangle$ with a sort for every object of $C$ and an operation for every morphism of $C$.
\end{definition}

\begin{definition}[Internal Theory]
Given a $\kappa$-coherent category $C$, its internal theory $Th(C)$ has as its signature $L_C$, the canonical signature of $C$ (Def \ref{definition:canonical_signature}), and contains sequents which refer to the identities, commutative triangles, finite limits, surjective arrows and $\kappa$-small unions of $C$.
\end{definition}

The internal theory of a $\kappa$-coherent category highlights the relationship between $\kappa$-coherent functors and models.

\begin{theorem}[Theorem 3.5.3 of \cite{makkai_first_1977}]
\label{theorem:model_categories_th}
The categories $Th(C)$-$mod(\mathcal{E})$ and $Coh(C, \mathcal{E})$ are isomorphic.
\end{theorem}

This relationship can also be inverted by noticing that every $\kappa$-coherent theory has an associated $\kappa$-coherent category, referred to as its syntactic category. The objects of this category are equivalence classes of formulae and its morphisms are equivalence classes of $T$-provably functional formulae.

\begin{definition}[$T$-provably functional formulae]
A formula $\theta(\vec{x}, \vec{y})$ is $T$-provably functional from $\phi(\vec{x})$ to $\psi(\vec{y})$ (notated as $\theta(\vec{x}, \vec{y}) : \phi(\vec{x}) \rightarrow \psi(\vec{y})$) if it meets the following conditions.
\begin{itemize}[noitemsep,topsep=-8pt]
    \item $T\vDash \theta(\vec{x}, \vec{y}) \vdash \phi(\vec{x}) \wedge \psi(\vec{y})$.
    \item $T\vDash \phi(\vec{x}) \vdash \exists \vec{y}.\theta(\vec{x}, \vec{y})$
    \item $T\vDash \theta(\vec{x}, \vec{y}) \wedge \theta(\vec{x}, \vec{y'}) \vdash \vec{y}=\vec{y'}$
\end{itemize}
\end{definition}

\begin{definition}[Syntactic Category]
Given a $\kappa$-coherent theory $T$, its syntactic category $Syn(T)$ is given by the following.
\begin{itemize}[noitemsep,topsep=-8pt]
   \item The objects $[\phi(\vec{x})]\in Syn(T)$ are equivalence classes of coherent formulae up to free variable substitution. $\phi(\vec{x}) \sim \psi(\vec{y}) \iff \psi(\vec{y}) = \phi(\vec{y}/\vec{x})$.
   \item The morphisms $[\theta(\vec{x}, \vec{y})]: [\phi(\vec{x}) ] \rightarrow [\psi(\vec{y}))]$ are equivalence classes of $T$-provably functional formulae. $\theta(\vec{x}, \vec{y})\sim\theta'(\vec{x}, \vec{y})\iff (T\vDash \theta(\vec{x}, \vec{y}) \vdash \theta'(\vec{x}, \vec{y}))\wedge(T\vDash \theta'(\vec{x}, \vec{y}) \vdash \theta(\vec{x}, \vec{y})) $
\end{itemize}
\end{definition}

\begin{theorem}
\label{theorem:model_categories}
The categories $T$-$mod(C)$ and $Coh(Syn(T),C)$ are equivalent, written $T$-$mod(C)\simeq Coh(Syn(T),C)$.
\end{theorem}

Finally, in order to provide all of the machinery for Thm \ref{theorem:hard_coding} it is necessary to define the concept of a replete subcategory (as an isofibration) and when a $\kappa$-coherent category is well pointed.

\begin{definition}[Isofibration \cite{johnson_2-dimensional_2020, shulman_isofibration_2023}]
An isofibration is a functor $p:E \rightarrow B$ such that for any object $e \in E$ and any isomorphism $i : p(e) \cong b$ there exists an isomorphism $j : e \cong e'$ such that $p(j) = i$.
\end{definition}

\begin{definition}[Replete Subcategory \cite{borceux_handbook_1994, bartels_replete_2024}]
A subcategory $D \rightarrowtail C$ is a replete subcategory if its inclusion is an isofibration.
\end{definition}

\begin{definition}[Well Pointed \cite{mac_lane_sheaves_1994, trimble_well-pointed_2023}]
A $\kappa$-coherent category is well pointed if its initial object is not its terminal object and its terminal object is an extremal generator: 
\begin{itemize}[noitemsep,topsep=-8pt]
   \item (Generator) Given $f,g: A \rightarrow B$ if $fx = gx$ for all $x: \mathbf{1} \rightarrow A$ then $f=g$.
   \item (Extremal) For any monomorphism $m:A\rightarrowtail B$, if every global element $\mathbf{1} \rightarrow B$ factors through $m$, then $m$ is an isomorphism.
\end{itemize}
\end{definition}

%% file: sections/datasets.tex
\section{Datasets}
\label{section:datasets}

Many machine learning problems may be characterised as error minimisation problems, where a model is selected that minimises the error between it and some target dataset. As shown in \cite{pugh_learning_2025}, such an error minimisation problem may be presented as a functor from a category of models into a category of datasets $Inf : M \rightarrow D$. This functor may be thought of as model inference. Usually, model inference is taken to be the evaluation of a machine learning model on a collection of inputs. By performing this process over the entire domain of a machine learning model one may construct a dataset. This induces a mapping from machine learning models to datasets. Abstracted to the category theoretic level, one may define such an operation directly as a functor, rather than through construction. Any such functor may be presented as an extension problem in some 2-category $\mathbb{T}$.

\begin{equation}
\begin{tikzcd}
	& \mu \\
	\delta && \tau
	\arrow["m", dashed, from=1-2, to=2-3]
	\arrow["\iota", from=2-1, to=1-2]
	\arrow["d"', from=2-1, to=2-3]
\end{tikzcd}
\end{equation}

The inference functor may be defined via the exponential $Inf = \mathbb{T}(\iota, \tau) : \mathbb{T}(\mu, \tau) \rightarrow \mathbb{T}(\delta, \tau)$, mapping models $m \in M = \mathbb{T}(\mu, \tau)$ via precomposition to datasets $m\iota \in D = \mathbb{T}(\delta, \tau)$. Presented as an extension with a 2-category, the error minimiser of such a problem is the left Kan-extension \cite{pugh_learning_2025}. Given this starting point, the question this work attempts to answer is how and where to construct such an extension problem for a given machine learning architecture. The answer stems from a particular perspective, that machine learning models are themselves datasets or databases. They are a method of storing data. The second step in the approach is to assert that datasets or databases are models of formal theories \footnote{The suggestion that databases should be models of formal theories is relatively standard within database theory \cite{abiteboul_foundations_1995}}. The consequence of this perspective is that an inference functor is given by an interpretation of the theory of datasets $\delta$ inside the theory of machine learning models $\mu$. An alternate phrasing would be to say that machine learning inference is determined by a canonical model of the chosen theory of datasets within the theory of machine learning models. 

This section will demonstrate how a relatively simple categorical model of databases, functorial databases \cite{spivak_functorial_2013}, can be lifted into a presentation as models of formal theories of a fragment of coherent logic. It will then develop tools for hard coding data into a $\kappa$-coherent category so that it may be used in the description of the relevant theory of machine learning models and datasets. The subsequent section will then demonstrate how to construct theories which have neural networks as their models, using these tools.

In the context of formal databases, the category of functorial database schema $Sch$ is equivalent to the category of small categories $Cat$. Consider the schema for the database of a shop, presented as the category below, called $Shop$.

\begin{equation}
\begin{tikzcd}
	Item & Price \\
	Order & Customer \\
	Employee & Person & Address
	\arrow["a", from=1-1, to=1-2]
	\arrow["b", from=2-1, to=1-1]
	\arrow["c", from=2-1, to=2-2]
	\arrow["d"', from=2-1, to=3-1]
	\arrow["e"', from=2-2, to=3-2]
	\arrow["f"', from=3-1, to=3-2]
	\arrow["g"', from=3-2, to=3-3]
\end{tikzcd}
\end{equation}

An instance or $Set$-model of this database schema is a functor $d:Shop \rightarrow Set$. An example $Set$-model can be seen in Fig \ref{figure:shop_data}.

\begin{figure}[htbp]      
    \centering
    \includegraphics[width=0.7\columnwidth]{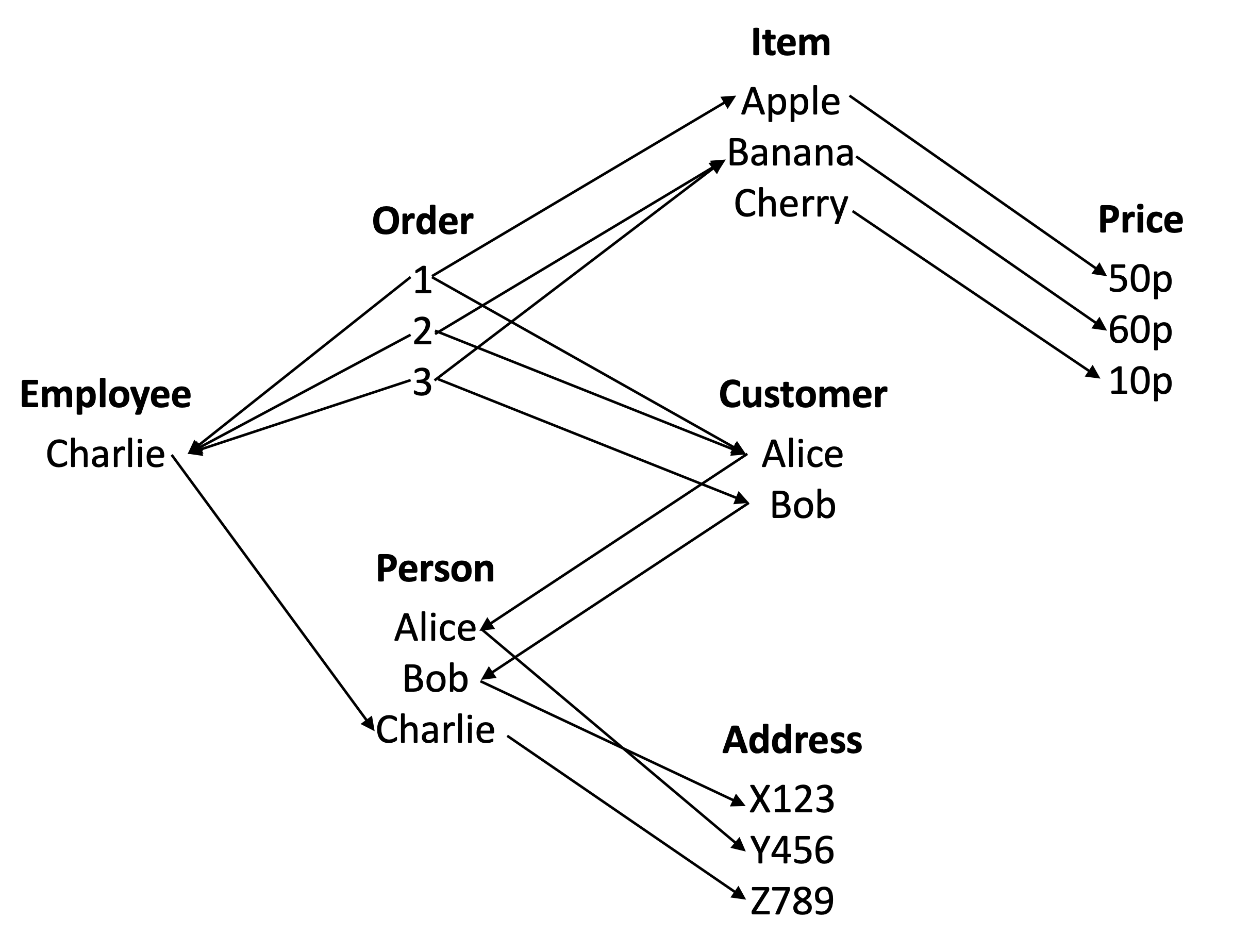}
    \caption{An example $Set$-model of $Shop$}
    \label{figure:shop_data}
\end{figure}

It is possible, given any such schema $D$, to find a $\kappa$-coherent theory whose category of models in any category $C$ is equivalent to the models of the functorial database schema in $C$.

\begin{theorem}
\label{theorem:schema}
For any small category $D$, there exists a $\kappa$-coherent theory, for any choice of $\kappa$, $Sch(D)$ such that for any $\kappa$-coherent category $C$, $Cat(D, C)$ is equivalent to $Coh(SynSch(D), C)$.
\end{theorem}
\begin{proof}
Construct a signature $L := \langle ob(D), mor(D), \varnothing \rangle$. Let $Sch(D)$ be a theory with signature $L$ and for any commutative triangle $fg = h$ in $D$, $Sch(D)$ contains the sequent $\top \vdash f(g(x)) = h(x)$. Given a coherent category $C$, by Def \ref{def:model} a $Sch(D)$-model, $M$, is a $C$-structure of type $L$ within which all sequents of $Sch(D)$ are valid. By Def  \ref{def:valid} a sequent $\top \vdash f(g(x)) = h(x)$ is valid if $M_{x}(\top) \leq M_{x}(f(g(x)) = h(x))$ which as $M(\top)$ is the top element of the subobject preorder implies that the sub-objects $e := eq(M(f)M(g), M(h)) : X \rightarrowtail M(X)$ and $Id_{M(x)} : M(X) \cong M(X)$ are related by an isomorphism $a : X \cong M(X)$ such that $e = a Id_{M(x)}$. This makes $e$ an isomorphism and thus an epimorphism. As an epimorphism, the relationship $eM(f)M(g)=eM(h)$ from the equaliser diagram implies $M(f)M(g)=M(h)$. Therefore a $Sch(D)$-model maps objects of $D$ to objects of $C$, morphisms of $D$ to morphisms of $C$, such that morphism composition is preserved, making $Sch(D)$-models functors of the form $M:D\rightarrow C$ and vice versa. This produces a bijection between $Ob(Cat(D,C))$ and $Ob(Sch(D)\text{-}Mod(C))$.

By Def \ref{def:model_morphisms} a morphism of models $\alpha: M \rightarrow M'$ is a family of morphisms $\alpha_s : M(s) \rightarrow M'(s)$ in $C$ for every sort $s$ in $L$ and thus over objects $s \in Ob(D)$ which satisfy the relevant commutative diagrams for all operations and predicates in $L$. Because $L$ has no relations, and only operations with a single input variable, then $\alpha$ is a family of morphism in $C$ indexed by the objects of $D$ and  required to satisfy commutative diagrams of the following form.
\begin{equation}
\begin{tikzcd}
	{M(s_i)} & {M(s_o)} \\
	{M(s_i)} & {M'(s_o)}
	\arrow["{M(f)}", from=1-1, to=1-2]
	\arrow["{\alpha_{s_i}}"', from=1-1, to=2-1]
	\arrow["{\alpha_{s_o}}", from=1-2, to=2-2]
	\arrow["{M'(f)}"', from=2-1, to=2-2]
\end{tikzcd}
\end{equation}
The diagrams are also the naturality square for the natural transform $\alpha : M \Rightarrow M'$ that the morphisms of $Cat(D, C)$ are required to satisfy. This not only puts the hom-objects $Sch(D)\text{-}Mod(C)(M, M')$ and $Cat(D,C)(M(f),M'(f))$ in bijection, but implies the isomorphism $Cat(D, C) \cong Sch(D) \text{-}Mod(C)$. From Thm \ref{theorem:model_categories},  $Sch(D)\text{-}Mod(C)\simeq Coh(SynSch(D), C)$. As an isomorphism is an equivalence and equivalences are transitive this results in $Cat(D, C) \simeq Coh(SynSch(D), C)$
\end{proof}

\begin{corollary}
There exists a logical fragment of coherent logic whose theories are functorial schema and whose models are functorial databases. The signature of theories of this logic have only sorts and operations, and the formulae are those formed from only atomic formulae and term substitution.
\end{corollary}

Thm \ref{theorem:schema} will be useful for forming coherent categories whose models may be identified with diagrams in the target category, allowing the production of certain (co)limit diagrams using these coherent functors. It will also be helpful to have a method of producing $\kappa$-coherent categories whose models are guaranteed to contain already known data which may be referenced by future constructions. It is not always convenient to construct an axiomatised representation of some data structure so that it may be referenced within models of the theory. For this purpose a 'Hard Coding' theorem (Thm \ref{theorem:hard_coding}) is introduced which guarantees for certain functors or functorial database instances there exist $\kappa$-coherent theories whose models all contain that data.

\begin{lemma}
\label{lemma:smallest_geo_sub}
Given a small, well-pointed, $\kappa$-coherent category $\tau$, and a functor $F:C \rightarrow \tau$ such that $\tau(\mathbf{1}, F(X))$ is $\kappa$-small, there exists a smallest replete $\kappa$-coherent subcategory $\varphi : \delta \rightarrow \tau$ that contains the image of $F$ and all morphisms $x:\mathbf{1} \rightarrow F(X)$ for any $X \in C$.
\end{lemma}
\begin{proof}
Consider the category $\Lambda$ whose morphisms are $\kappa$-coherent inclusion functors and whose objects are replete $\kappa$-coherent subcategories of $\tau$ which contain the image of $F$ and any morphism $x:\mathbf{1} \rightarrow F(X)$ for any $X \in C$. This category is not empty because the identity functor $Id_\tau : \tau \rightarrow \tau$ satisfies these conditions.

Because $\tau$ is a small category, its morphisms are contained in a small set, meaning that the power-set of its morphisms $\mathcal{P}(Mor(\tau))$ is also a small set. Any object of $\Lambda$ represent a subset of the morphisms of $\tau$, producing an inclusion $Ob(\Lambda) \rightarrowtail \mathcal{P}(Mor(\tau))$ meaning that $Ob(\Lambda)$ must also be a small set. Because the morphisms of $\Lambda$ are inclusion functors then $\Lambda$ is a preorder which, as $Ob(\Lambda)$ is small, means that $\Lambda$ is a small category. As $\Lambda$ is a small subcategory of $Cat$ it represents a small diagram. Because $Cat$ contains all small limits then the limit of $\Lambda$ exists in $Cat$. Call this limit $\varphi : \delta \rightarrow \tau$. This limit represents the intersection of replete $\kappa$-coherent subcategories, making $\delta$ a replete $\kappa$-coherent subcategory of all of the categories in $\Lambda$, possessing an inclusion functor $\delta \rightarrow \delta_i$ for any $\delta_i\in\Lambda$ making $\delta$ smaller than all the categories in $\Lambda$. Because all the categories in $\Lambda$ contain the image of $F$ and any morphism $x:\mathbf{1} \rightarrow F(X)$ for any $X \in C$ then so does $\varphi: \delta \rightarrow \tau$.
\end{proof}

\begin{corollary}
Given a small, well-pointed, infinitary-coherent category $\tau$, and a functor $F:C \rightarrow \tau$ there exists a smallest replete infinitary-coherent subcategory $\varphi : \delta \rightarrow \tau$ that contains the image of $F$ and all morphisms $x:\mathbf{1} \rightarrow F(X)$ for any $X \in C$.
\end{corollary}

\begin{lemma}
\label{lemma:lift_from_C_to_delta}
There exists a functor $\rho: C \rightarrow \delta$ such that $\varphi\rho = F$
\end{lemma}

\begin{proof}
This is a direct consequence of $\varphi$ being an inclusion functor and $F$ being entirely contained in the image of $\varphi$.
\end{proof}

\begin{lemma}
\label{lemma:isomorphic_component}
For any $\kappa$-coherent functor $\varrho : \delta \rightarrow \tau$ there exists an isomorphism $\delta(\mathbf{1}, \rho(X)) \cong \tau(\mathbf{1}, \varrho \rho(X))$ for all $X\in C$.
\end{lemma}

\begin{proof}
Consider the component function $\varrho_{\mathbf{1}, \rho(X)} : \delta(\mathbf{1}, \rho(X)) \rightarrow \tau(\varrho(\mathbf{1}),\varrho \rho(X))$. First, demonstrate that this function is a monomorphism. Given two morphisms $x,y:\mathbf{1} \rightarrow \rho(X)$ in $\delta$, such that $x \neq y$, notice that the pull-back of $x$ along $y$ produces the initial object $\varnothing \in \delta$. Because $\varrho$ is a $\kappa$-coherent functor, it preserves finite limits, meaning that the pull-back of $\varrho(x)$ along $\varrho(y)$ in $\tau$ is $\varrho(\varnothing)$. Because $\varrho$ preserves  $\kappa$-small unions, including the empty union, $\varrho(\varnothing)$ is the initial object in $\tau$, $\varrho(\varnothing) \cong \varnothing$. Because $\tau$ is well-pointed, the initial object is distinct from the terminal object. As the pull-back produces the initial object and not the terminal object, this implies that $\varrho(x) \neq \varrho(y)$ if $x\neq y$ showing that $\varrho_{\mathbf{1}, \rho(X)}$ is a monomorphism.

Second, demonstrate that $\varrho_{\mathbf{1}, \rho(X)}$ is an epimorphism. Because $\tau$ is well pointed, $\mathbf{1}$ is an extremal generator, meaning that the union $\bigvee \{z:\mathbf{1} \rightarrow F(X)\}$ is isomorphic to $F(X)$. As shown in Lemma \ref{lemma:smallest_geo_sub}, $\delta$ contains all morphisms $x : \mathbf{1} \rightarrow F(X)$. Using the formula $F=\varphi\rho$ from Lemma \ref{lemma:lift_from_C_to_delta}, this can be rephrased as the statement that $\bigvee \{\varphi(x):\mathbf{1} \rightarrow \varphi\rho(X)\}$ is isomorphic to $\varphi\rho(X)$. Because $\varphi$ is $\kappa$-coherent, it preserves $\kappa$-small unions and by assumption $\{x:\mathbf{1} \rightarrow \rho(X)\}$ is $\kappa$-small. $\varphi(\bigvee \{x:\mathbf{1} \rightarrow \rho(X)\}) \cong \varphi\rho(X)$. Because $\delta$ is a replete subcategory, $\varphi$ is an isofibration, meaning that it lifts isomorphisms, leading to the statement that $\bigvee \{x:\mathbf{1} \rightarrow \rho(X)\}) \cong \rho(X)$. Because $\varrho$ is a $\kappa$-coherent functor, it preserves $\kappa$-small unions, $\varrho \rho(X) \cong \bigvee\{\varrho(x):\mathbf{1} \rightarrow \varrho\rho(X)\}$.

If $\varrho_{\mathbf{1}, \rho(X)}$ was not an epimorphism, there would exist a $y:\mathbf{1} \rightarrow \varrho \rho(x)$ such that there would be no $x : \mathbf{1} \rightarrow \rho(x)$ in $\delta$ where $\varrho(x) = y$. This would imply that the pull-back of $\bigvee\{\varrho(x):\mathbf{1} \rightarrow \varrho\rho(X)\}$, as a sub-object of $\varrho \rho(X)$, along $y$, would be the initial object $\varnothing$. However, it was shown that the union is isomorphic to $\varrho \rho(x)$ and the pull-back of $\varrho \rho(x)$ along $y$ cannot be the initial object, implying contradiction and demonstrating that no such $y$ can exist. This demonstrates that $\varrho_{\mathbf{1}, \rho(X)}$ must be epimorphic. As $\varrho_{\mathbf{1}, \rho(X)}$ is a set function which is both monomorphic and epimorphic, it must be isomorphic.
\end{proof}

\begin{theorem}[Hard Coding]
\label{theorem:hard_coding}
Given a functor $F:C \rightarrow \tau$ where $\tau$ is a small well-pointed $\kappa$-coherent category, and $\tau(\mathbf{1}, F(X))$ is $\kappa$-small, there exists a $\kappa$-coherent category $\delta$ with functor $\rho: C\rightarrow \delta$ such that for any $\kappa$-coherent functor $\varrho: \delta \rightarrow \tau$ there exists a natural isomorphism $F \cong \varrho\rho$.
\end{theorem}

\begin{proof}
Use lemma \ref{lemma:smallest_geo_sub} to construct $\varphi : \delta \rightarrow \tau$. From lemma \ref{lemma:lift_from_C_to_delta} there exists a functor $\rho : C \rightarrow \delta$ such that $\varphi \rho = F$. Because $\varphi$ is faithful, then for any $X \in C$ the component function $\varphi_{\mathbf{1}, \rho(X)} : \delta(\mathbf{1}, \rho(X)) \rightarrow \tau(\varphi(\mathbf{1}), \varphi\rho(X))$ is a monomorphism. By construction, as every morphism $x:\mathbf{1} \rightarrow F(X)$ is contained in the image of $\varphi$, then the component morphism $\varphi_{\mathbf{1},\rho(X)}$ is an epimorphism. As $\varphi_{\mathbf{1},\rho(X)}$ is a set function which is both monomorphic and epimorphic, then it is isomorphic, meaning $\delta(\mathbf{1}, \rho(X)) \cong \tau(\varphi(\mathbf{1}), \varphi\rho(X))$. For any $\kappa$-coherent functor $\varrho: \delta \rightarrow \tau$, by lemma \ref{lemma:isomorphic_component} there exists an isomorphism $\delta(\mathbf{1}, \rho(X)) \cong \tau(\mathbf{1}, \varrho \rho(X))$. By combination with the previous isomorphism, one produces the following deduction.
\begin{small}
\begin{equation}
\tau(\mathbf{1}, \varphi\rho(X)) \cong \tau(\varphi(\mathbf{1}), \varphi\rho(X)) \cong \delta(\mathbf{1}, \rho(X))\cong \tau(\mathbf{1}, \varrho \rho(X))
\end{equation}
\end{small}

Because $\mathbf{1}$ is an extremal separator in $\tau$, then $\bigvee \{ x:\mathbf{1} \rightarrow \varrho\rho(X)\} \cong \varrho\rho(X)$ and $\bigvee \{ x:\mathbf{1} \rightarrow \varrho\rho(X)\} \cong \varphi\rho(X)$. Because $\tau(\mathbf{1}, \varphi\rho(X)) \cong \tau(\mathbf{1}, \varrho \rho(X))$ one may re-index the union, to form the following deduction.
\begin{small}
\begin{equation}
\varphi\rho(X) \cong \bigvee \{ x:\mathbf{1} \rightarrow \varphi\rho(X)\} \cong \bigvee \{ x:\mathbf{1} \rightarrow \varrho\rho(X)\} \cong \varrho\rho(X)
\end{equation}
\end{small}
Where the isomorphism corresponds to the following map.
\begin{equation}
(\varphi(x) : \mathbf{1} \rightarrow \varphi\rho(X)) \mapsto (\varrho(x) : \mathbf{1} \rightarrow \varrho\rho(X))
\end{equation}

This map can be used to show that the following diagram commutes through a diagram chase.

\begin{equation}
\begin{tikzcd}[column sep=large]
	{\varphi\rho(X)} && {\varphi\rho(Y)} \\
	& {\mathbf{1}} \\
	{\varrho\rho(X)} && {\varrho\rho(Y)}
	\arrow[""{name=0, anchor=center, inner sep=0}, "{\varphi\rho(f)}", from=1-1, to=1-3]
	\arrow[""{name=1, anchor=center, inner sep=0}, "\sim"', from=1-1, to=3-1]
	\arrow[""{name=2, anchor=center, inner sep=0}, "\sim", from=1-3, to=3-3]
	\arrow["{\varphi(x)}"{description}, from=2-2, to=1-1]
	\arrow["{\varphi(\rho(f)x)}"{description}, from=2-2, to=1-3]
	\arrow["{\varrho(x)}"{description}, from=2-2, to=3-1]
	\arrow["{\varrho(\rho(f)x)}"{description}, from=2-2, to=3-3]
	\arrow[""{name=3, anchor=center, inner sep=0}, "{\varrho\rho(f)}"', from=3-1, to=3-3]
	\arrow["{\circlearrowleft_4}"{description}, draw=none, from=2-2, to=0]
	\arrow["{	\circlearrowleft_1}"{description, pos=0.7}, draw=none, from=2-2, to=1]
	\arrow["{\circlearrowleft_3}"{description, pos=0.7}, draw=none, from=2-2, to=2]
	\arrow["{\circlearrowleft_2}"{description}, draw=none, from=2-2, to=3]
\end{tikzcd}
\end{equation}

The cells $\circlearrowleft_1$ and $\circlearrowleft_3$ commute due to the construction of the isomorphism mapping. The cells $\circlearrowleft_2$ and $\circlearrowleft_4$ commute because $\varphi$ and $\varrho$ are functors. Because $\mathbf{1}$ is a generator, and the perimeter morphisms commute for every element, then they must commute in general. Because each of the isomorphisms may also be written as $F(X) \cong \varrho\rho(X)$, they form a natural family of isomorphisms, with the perimeter of the diagram being the naturality square for the natural isomorphism $F \cong \varrho\rho$.
\end{proof}

\begin{corollary}
Given a functor $F:C \rightarrow \tau$ where $\tau$ is a well-pointed infinitary coherent category, there exists an infinitary coherent category $\delta$ with functor $\rho: C\rightarrow \delta$ such that for any geometric functor $\varrho: \delta \rightarrow \tau$ there exists a natural isomorphism $F \cong \varrho\rho$.
\end{corollary}

%% file: sections/ml_models.tex
\section{Machine Learning Models}
\label{section:mlmodels}

To those familiar with formal theories and model theories, it will appear intuitively true that any machine learning algorithm may be represented as a model of formal theories. Formal systems and type theories provide the language in which any mathematically rigorous statement is made, suggesting that if one can describe a machine learning architecture rigorously, then it is describable (and is a model) of a theory in some formal system. An even more practical intuition may be derived by drawing on the connection between type theories and programming languages \cite{coecke_physics_2010} by suggesting that any programmatic implementation of machine learning algorithm is already a description of it within a formal theory. The intention of this section is to provide one concrete methodology by which a class of machine learning architectures, those describable as feed forward networks with tied or fixed weights \footnote{This is a surprisingly wide reaching subset of machine learning algorithms capable of describing dense layers, sparse layers, convolutional layers, kernel svms, etc.}, are presented as models of a coherent theory. Model inference is then presented as a functor which is induced via the precomposition of an interpretation between formal theories. The processes presented in this section are demonstrated for implementations of neural network which utilise floating point arithmetic \footnote{A finite approximation of a portion of real number arithmetic.}. This is partly because the following constructions utilise (co)limit constructions in $Coh_\sim$ requiring a proof of their existence (Thm \ref{theorem:bicomplete})\footnote{It should also be noted that implementations of machine learning algorithms typically use floating point arithmetic rather than an exact real arithmetic. These implementations produce demonstrably different behaviour than what is predicted by theory. On principle, it may be worth considering modifying the standard practice in the literature such that theory more directly reflects the experimental setup. This could be done by modifying the definitions given in papers such that the (co)domains of functions are defined relative to the set of floating point numbers, or by testing algorithms using exact real arithmetic implementations. One possible compromise may be to define and test ML algorithms using the rational numbers, where both the theory and exact implementations are practical to work with.}. The hard coding of the information in the arithmetic requires a $\kappa$-coherent theory for a $\kappa$ small set, limiting to the construction of finite sets until a proof for the bi-completeness of $\kappa\text{-}Coh_\sim$ is found\footnote{Such a proof is likely to be found as a direct generalisation of that shown in \cite{kanalas_21-category_2021} but is beyond the scope of this work}.

\begin{lemma}
\label{lemma:float}
There exists a coherent category $Float\text{-}\sigma$ of floating point arithmetic with an activation function $\sigma$, whose models in $Set$ contain a model of floating point arithmetic and the activation function.
\end{lemma}

\begin{proof}
Consider the category $Pol := \{V\xleftarrow[]{s}E\xrightarrow[]{t}V\xrightarrow[]{a}V\}$. Fix a model of the floating point numbers $R$ for some number of bits and a choice of activation function. Define a functor $F : Pol \rightarrow Set$ such that $F(V)=R$,  $F(E) = R\times R$, $F(s) = add : R \times R \rightarrow R$, $F(t) = mul : R \times R \rightarrow R$, and $F(a) = \sigma : R \rightarrow R$. Set is a well-pointed coherent category, and $R$ is a finite set making $Set(\mathbf{1}, F(X))$ finite, $\aleph_0$-small, implying  by Thm \ref{theorem:hard_coding} that there exists a coherent category $Float\text{-}\sigma$, a coherent functor $\psi : Float\text{-}\sigma \rightarrow Set$, and a functor $\rho : Pol \rightarrow Float\text{-}\sigma$ such that $\psi\rho = F$.
\end{proof}

\begin{corollary}
\label{corollary:float}
When $\sigma = Id_R$, then Lemma \ref{lemma:float} produces a coherent category for floating point numbers $Float$ with the obvious coherent functor $[\sigma] : Float \rightarrow Float\text{-}\sigma$ for any other choice of $\sigma$.
\end{corollary}

To be clear about the objective of the construction, consider the following definition of a densely connected neural network architecture. Such an architecture is a composition of linear transformations and dimension wise applications of activation functions.

\begin{definition}[Dense Neural Network Architecture]
\label{definition:nna}
An architecture of depth $n$ is defined by an $n$-tuple $\vec{\sigma} = (\sigma_1, ... , \sigma_n)$ of possibly non-linear activation functions $\sigma_i : R \rightarrow R$ and an $(n+1)$-tuple $\vec{a} = (a_0, ..., a_n)$ of natural numbers. Then $NN(\vec{\sigma}, \vec{a})$ is the set of pairs $(\vec{w},\vec{\beta})$ where $\vec{w}$ is an $n$-tuple $\vec{w} = (w_1, ..., w_n)$ of weight matrices $w_i : R^{a_{i-1}} \rightarrow R^{a_i}$ and $\vec{\beta}$ is an $n$-tuple of bias vectors $\beta_i : \mathbf{1} \rightarrow R^{a_i}$.
\end{definition}

Weight fixing of a network is the assertion that some of its weights must take specific values. This prevents models of the network architecture from modifying this weight. A sparsely connected neural network may be modelled by fixing certain weights as zero. Weight tying is the assertion that two weights in the network must take the same value. Models such as convolution layer, which produce a linear transform utilising a convolution kernel applied over a space of inputs, can be represented as a dense layer with weights tied corresponding to the convolution kernel. This can be seen intuitively by realising that any linear map between finite vector spaces may be presented as a matrix and that the convolution action asserts that various sub matrices within the linear transform are the same. The information associated with both weight fixing and weight tying may be encoded through pairs of weights, biases, or values which represent the equalities satisfied by the weights and biases.

To represent the inference function as a functor, it is also important to present a rigours description of the output datasets. The data associated with a supervised training set is an indexed collection of pairs of inputs and outputs $\{(f(i),t(i))\}_{i\in N}$. A more categorical description may present this information as a span.

\begin{definition}[R-Span Datasets]
The set $RSpanData(n,m)$ is the set of all set function spans for finite $N$.\[R^n\xleftarrow[]{f}N\xrightarrow[]{t}R^m\]
\end{definition}

The inference function computes a dataset from trained machine learning models (models of the architecture). This dataset simply represents the neural networks evaluation at every point in its domain.

\begin{definition}[Inference Function]
\label{definition:inference}
For a given architecture $NN(\vec{\sigma}, \vec{a})$, the inference function $Inf : NN(\vec{\sigma}, \vec{a}) \rightarrow RSpanData(\vec{a}_0,\vec{a}_n)$ maps a weight $n$-tuple $\vec{w}$ to the following span. \[R^n\xleftarrow[]{Id}R^n\xrightarrow[]{t}R^m\] Where $t(x) = \sigma_n(w_n...\sigma_1(w_1x+\beta_1)+\beta_n)$.
\end{definition}

By the definition of the architecture as a list of transfer functions and layer widths, the composition of architectures may be described directly through the concatenation of lists. As these lists are generated by the collection of singleton lists, the space of dense neural networks may be produced directly for single layer networks (networks with a one weight matrix, i.e. the length of $\vec{a}$ is 2).

\begin{definition}[Neural Network Composition]
Given two architectures $NN(\vec{\sigma}, \vec{a})$ and $NN(\vec{\sigma}', \vec{a}')$ of lengths $n$ and $m$ such that $a'_0 = a_n$ then the composite $NN(\vec{\sigma}, \vec{a})\circ NN(\vec{\sigma}', \vec{a}')$ is the architecture $NN(\vec{\sigma}'', \vec{a}'')$ of length $n+m$, where $\vec{\sigma}''$ and $\vec{a}''$ are defined as follows.
\begin{gather}
\sigma''_i = \begin{cases} 
      \sigma_i & 1\leq i\leq n \\
      \sigma'_{i-n} & n<i\leq n+m \\
   \end{cases}\\
a''_i = \begin{cases} 
      a_i & 0\leq i\leq n \\
      a'_{i-n} & i>n \\
   \end{cases}
\end{gather}
\end{definition}

The construction of architectures as formal theories will stitch the basic components together using (co)limits of $Coh_\sim$. This can be done by identifying diagrams inside relevant categories which are to be equated. However, the objects of $Coh_\sim$ are coherent categories, so the relevant diagrams shapes need to correspond with some coherent category. This is where Thm \ref{theorem:schema} may be used to construct coherent categories whose coherent functors may identify diagrams from other coherent categories. The following definitions name such categories which represent particularly useful diagram shapes.

\begin{definition}[Sort]
\label{definition:category_sort}
Let $Sort$ be the coherent category \[SynSch(\mathbf{1})\]
\end{definition}

\begin{definition}[Oper]
\label{definition:category_oper}
Let $Oper$ be the coherent category \[SynSch(\{A\xrightarrow[]{g}B\})\]
\end{definition}

\begin{definition}[Span]
\label{definition:category_span}
Let $Span$ be the coherent category \[SynSch(\{X\xleftarrow[]{f}N\xrightarrow[]{t}Y\})\]
\end{definition}

An R-Span datasets is merely a span whose codomains are self products of the model of a given model of the floating point numbers. To construct a coherent theory such that its functors are models of R-Span datasets, one must make sure it carries or contains the data of both $Span$ and $Float$. This can be done by fusing the two categories together with a colimit.

\begin{lemma}
A coherent category $RSpan(n,m)$ exists whose set models contain R-Span Datasets.
\end{lemma}

\begin{proof}
By construction (Lemma \ref{lemma:float}), the functor $\varphi\rho : Pol \rightarrow Set$ factors through the functor $\rho : Pol \rightarrow Float\text{-}\sigma$ and the coherent functor $\varphi : Float\text{-}\sigma \rightarrow Set$ where $\varphi\rho(V) = R$ is the set of floating point numbers. In this context, $\rho(V)$ can be interpreted as an internal object of floating point numbers of $Float\text{-}\sigma$. As coherent categories contain all limits and there exists $\rho(V)\in Float\text{-}\sigma$ then there also exists $\rho(V)^n\in Float\text{-}\sigma$. By Def \ref{definition:category_sort} and Thm \ref{theorem:schema}, $Float\text{-}\sigma \simeq Coh(Sort, Float\text{-}\sigma)$, which implies the existence of coherent functors $\rho(V)^n,\rho(V)^m  : Sort \rightarrow Float\text{-}\sigma$ determined up to isomorphism. Similarly, by Def \ref{definition:category_span}, there exists objects $X, Y \in Span$, which implies the existence of coherent functors $X, Y: Sort \rightarrow Span$. These coherent functors may be used to form the $2$-colimit in $Coh_\sim$ by Thm \ref{theorem:bicomplete} (The 2-colimit diagram commutes and factors up to 2-cell isomorphism).

\begin{equation}
\begin{tikzcd}
	& Sort \\
	Sort && {Float} \\
	& Span & {RSpan(n,m)}
	\arrow["{\rho(V)^m}", from=1-2, to=2-3]
	\arrow["Y"{pos=0.7}, from=1-2, to=3-2]
	\arrow[""{name=0, anchor=center, inner sep=0}, "{\rho(V)^n}"{pos=0.7}, from=2-1, to=2-3]
	\arrow["X"', from=2-1, to=3-2]
	\arrow["{i_0}", from=2-3, to=3-3]
	\arrow["{i_1}", from=3-2, to=3-3]
	\arrow["\ulcorner"{anchor=center, pos=0.125, rotate=180}, draw=none, from=3-3, to=0]
\end{tikzcd}
\end{equation}

Consider a set model of $RSpan(n,m)$ called $M$. By Thm \ref{theorem:model_categories_th} $Th(RSpan(n,m))\text{-}mod(Set)$ is isomorphic to $Coh(RSpan(n,m), Set)$ allowing $M$ to be identified with a coherent functor $M : RSpan(n,m) \rightarrow Set$. The functor $i_1 : Span \rightarrow RSpan(n, m)$ induces a functor $Coh(i_1, Set) : Coh(RSpan(n,m), Set) \rightarrow Coh(Span(n, m), Set)$ showing that any $RSpan(n,m)$-model induces a $Span$-model. Similarly, any $RSpan(n,m)$-model induces a $Float\text{-}\sigma$-model. By Corollary \ref{corollary:float} and Thm \ref{theorem:hard_coding}, every $Float\text{-}\sigma$ model has a natural isomorphism to the floating point numbers, therefore $Mi_0\rho(V)^n \cong R^n$ and $Mi_0\rho(V)^m \cong R^m$. As the colimit diagram commutes up to isomorphism, this also implies that $Mi_1X \cong R^n$ and $Mi_1Y \cong R^m$ demonstrating that the $Span$-model induced by $M$ is an R-Span Dataset.
\end{proof}

Having constructed a suitable coherent theory for R-span datasets, the approach to constructing all neural network architectures will be to construct a suitable definition for a single layer network and demonstrate how layers may be composed. As any dense feed forward neural network is a composition of single layers this will demonstrate that all dense neural networks may be presented as models of coherent theories. An appropriate model of single layer networks must also admit a suitable interpretation of $RSpan(n,m)$ through which the inference functor may be induced. Both of these constraints may be satisfied at once by the following lemma.

\begin{lemma}
\label{lemma:coherent_nnl}
The inference functor for a floating point dense neural network layer may be represented as the exponential $Coh(\iota_{(\sigma,n,m)}, Set)$ of a coherent functor $\iota_{(\sigma,n,m)} : RSpan(n,m) \rightarrow \sigma NNL(n, m)$.
\end{lemma}

\begin{proof}
Construct $\sigma NNL(n,m)$ as the following 2-colimit, by Thm \ref{theorem:bicomplete}.
\begin{equation}
\begin{tikzcd}
	Sort && One && Sort \\
	Oper && {Float\text{-}\sigma} && Oper \\
	&& {\sigma NNL(n,m)}
	\arrow["B"', from=1-1, to=2-1]
	\arrow["{(\rho(V)^n)^m}"{pos=0.1}, from=1-1, to=2-3]
	\arrow["A"{pos=0.8}, from=1-3, to=2-1]
	\arrow["{\mathbf{1}}"{pos=0.4}, from=1-3, to=2-3]
	\arrow["A"'{pos=0.8}, from=1-3, to=2-5]
	\arrow["{\rho(V)^m}"'{pos=0.1}, from=1-5, to=2-3]
	\arrow["B", from=1-5, to=2-5]
	\arrow["w", from=2-1, to=3-3]
	\arrow["j", from=2-3, to=3-3]
	\arrow["\beta", from=2-5, to=3-3]
	\arrow["\ulcorner"{anchor=center, pos=0.125, rotate=180}, draw=none, from=3-3, to=1-1]
	\arrow["\ulcorner"{anchor=center, pos=0.125, rotate=90}, draw=none, from=3-3, to=1-5]
\end{tikzcd}
\end{equation}

The coherent functor $j : Float\text{-}\sigma \rightarrow \sigma NNL(n,m)$ induces a $Float\text{-}\sigma$ model in $\sigma NNL(n,m)$. For simplicity, refer to the operations and sorts of the induced float model via the functions and sets they represent. $j\rho(v) = R$, $j\rho(E) = R^2$, $j\rho(s) = add$, $j\rho(t) = mul$, and $j\rho(a) = \sigma$. Refer to the n-times self-categorical product of an object such as $R$ as $R^n$. For a morphism $f:X \rightarrow Y$ in a coherent category such as $\sigma NNL(n,m)$, there exists a canonical n-times product $f^n:= \langle f\pi_1, ..., f\pi_n \rangle: X^n \rightarrow Y^n$. For any object, such as $R$, There also exists an n-times copy map $\Delta_n := \langle Id_R, ..., Id_R \rangle : R \rightarrow R^n$. By observing that $R^n \cong R \times R^{n-1}$ one may infer the projection morphisms $\pi_1: R^n \rightarrow R$ and $\pi_{1<} : R^n \rightarrow R^{n-1}$. The operation $add_n$ can be defined inductively for $n \geq 1$ by stating that $add_1 = Id_R$ and using the following induction rule.
\begin{equation}
add_n(x) = add(\pi_1(x), add_{n-1}\pi_{1<}(x))
\end{equation}
By Def \ref{definition:category_oper} and Thm \ref{theorem:schema} the coherent functors $w : Oper \rightarrow \sigma NNL(n,m)$ and $\beta : Oper \rightarrow \sigma NNL(n,m)$ correspond to morphisms $w: \mathbf{1} \rightarrow (R^n)^m$ and $\beta : \mathbf{1} \rightarrow R^m$. Using these morphisms, the 'linear' map $L_{n,m} : R^n \rightarrow R^m$ may be defined by composition.
\begin{equation}
L_{(w,\beta)}(x) = add^m((add_n)^m (mul^n)^m(w, \Delta_m x),\beta)
\end{equation}
Making use of the transfer function $\sigma : R \rightarrow R$, the linear map may be used to produce the neural network layer $\sigma^m L_{(n,m)} : R^n \rightarrow R^m$. The neural network layer and the identity morphism $Id_{R^n} : R^n \rightarrow R^n$ represent a span diagram in $\sigma NNL(n,m)$. 
\begin{equation}
R^n\xleftarrow[]{Id}R^n\xrightarrow[]{\sigma^m L_{(w,\beta)}}R^m
\end{equation}
By Def \ref{definition:category_span} and Thm \ref{theorem:schema} the span diagram may be identified with the coherent functor $r_{(\sigma,n,m)} : Span \rightarrow \sigma NNL(n, m)$ which, in conjunction with the coherent functor $j$ (and $[\sigma]$ from Cor \ref{corollary:float}) produces a co-cone of the following 2-colimit diagram.
\begin{equation}
\begin{tikzcd}
	& Sort \\
	Sort && Float \\
	& Span & {RSpan(n,m)} \\
	&& {\sigma NNL(n,m)}
	\arrow["{\rho(V)^m}", from=1-2, to=2-3]
	\arrow["Y"{pos=0.7}, from=1-2, to=3-2]
	\arrow[""{name=0, anchor=center, inner sep=0}, "{\rho(V)^n}"{pos=0.7}, from=2-1, to=2-3]
	\arrow["X"', from=2-1, to=3-2]
	\arrow["{i_0}", from=2-3, to=3-3]
	\arrow["{j[\sigma]}", shift left=5, curve={height=-30pt}, from=2-3, to=4-3]
	\arrow["{i_1}", from=3-2, to=3-3]
	\arrow["{r_{(\sigma,n,m)}}"', from=3-2, to=4-3]
	\arrow["{\iota_{(\sigma,n,m)}}", from=3-3, to=4-3]
	\arrow["\ulcorner"{anchor=center, pos=0.125, rotate=180}, draw=none, from=3-3, to=0]
\end{tikzcd}
\end{equation}
By the universal factoring property of 2-colimits, the co-cone induced by $r_{(\sigma,n,m)}$ and $j$ factors through the coherent functor $\iota_{(\sigma,n,m)}$. The functor $\iota_{(\sigma,n,m)}$ induces the mapping $Inf:= Coh(\iota_{(\sigma,n,m)}, Set) : Coh(\sigma NNL(n,m), Set) \rightarrow Coh(RSpan(n,m), Set)$ via pre-composition, sending a neural network layer with the parameters $w$ and $\beta$ to the following R-span dataset 
\begin{equation}
R^n\xleftarrow[]{Id}R^n\xrightarrow[]{\sigma^m L_{(w,\beta)}}R^m \qedhere
\end{equation}
\end{proof}

The composition of neural network architectures may be found through the composition of spans via pullback. Such a definition of composition is possible in the abstract as long as a category contains a model of $RSpan(n,m)$ so it is possible to produce a lemma with slightly more generality than necessary.

\begin{lemma}[Layer Composition]
\label{lemma:composition}
Given coherent functors $\iota : RSpan(n,m) \rightarrow G$ and $\iota' : RSpan(m,k) \rightarrow H$, then there exists a coherent functor $\iota' \circ \iota : RSpan(n,k) \rightarrow H\circ G$ corresponding to the pull back of the component spans.
\end{lemma}

\begin{proof}
Form the following 2-colimit, by Thm \ref{theorem:bicomplete}.
\begin{equation}
\begin{tikzcd}
	Float && {RSpan(n,m)} \\
	&& G \\
	{RSpan(m,k)} & H & {H\circ G}
	\arrow[from=1-1, to=1-3]
	\arrow[from=1-1, to=3-1]
	\arrow["\iota"', from=1-3, to=2-3]
	\arrow["j"', from=2-3, to=3-3]
	\arrow["{\iota'}", from=3-1, to=3-2]
	\arrow["{j'}", from=3-2, to=3-3]
	\arrow["\ulcorner"{anchor=center, pos=0.125, rotate=180}, draw=none, from=3-3, to=1-1]
\end{tikzcd}
\end{equation}

By Def \ref{definition:category_span} and Thm \ref{theorem:schema} the composites $j\iota$ and $j'\iota'$ correspond to spans in $H\circ G$. Form the composite of spans via the following pullback.

\begin{equation}
\begin{tikzcd}
	{j\iota(N)\circ j'\iota'(N)} & {j'\iota'(N)} & {R'^k} \\
	& {R'^m} & {R^k} \\
	{j\iota(N)} & {R^m} \\
	{R^n} \\
	\arrow["{\pi_2}", from=1-1, to=1-2]
	\arrow["{\pi_1}"', from=1-1, to=3-1]
	\arrow["\ulcorner"{anchor=center, pos=0.125}, draw=none, from=1-1, to=3-2]
	\arrow["{j'\iota'(t)}", from=1-2, to=1-3]
	\arrow["{j'\iota'(f)}", from=1-2, to=2-2]
	\arrow[squiggly, tail reversed, from=1-3, to=2-3]
	\arrow[squiggly, tail reversed, from=2-2, to=3-2]
	\arrow["{j\iota(t)}"', from=3-1, to=3-2]
	\arrow["{j\iota(f)}"', from=3-1, to=4-1]
\end{tikzcd}
\end{equation}

It should be noted that as the 2-colimit commutes up to 2-cell isomorphism, the models of floating point real numbers are related by natural isomorphisms, so the endpoints of the composite span are isomorphic to $R^n$ and $R^k$ respectively.

It remains to show that this data produces a coherent functor $\iota' \circ \iota : RSpan(n,k) \rightarrow H \circ G$. Recall that $RSpan(n,k)$ is itself defined as a 2-colimit (in the proof of the $RSpan$ lemma), so a coherent functor out of $RSpan(n,k)$ corresponds to a co-cone of its defining diagram. That diagram requires a $Float$ model and a $Span$ model whose endpoints are the $n$-fold and $k$-fold self-products of the floating point object. Both are present in $H \circ G$: the $Float$ model is inherited from $G$ via $j$ and the colimit inclusion $j : G \rightarrow H \circ G$, and the composite span above, with endpoints isomorphic to $R^n$ and $R^k$, provides the required $Span$ model. These constitute a co-cone of the defining diagram of $RSpan(n,k)$, and the universal property of that 2-colimit yields the coherent functor $\iota' \circ \iota : RSpan(n,k) \rightarrow H \circ G$.
\end{proof}

Finally, with definitions for R-span datasets, dense layers, inference and composition, all through their descriptions as coherent categories, it is possible to prove the following theorem.

\begin{theorem}
\label{theorem:neural_network_representation}
For any dense floating point neural network architecture $NN(\vec{\sigma}, \vec{a})$ (Def \ref{definition:nna}) there exists a coherent functor $\iota : RSpan(a_0,a_n) \rightarrow G$ such that the $Set$ models of $G$ are neural networks of $NN(\vec{\sigma}, \vec{a})$, and such that the functor $Coh(\iota, Set): Coh(G, Set) \rightarrow Coh(RSpan(a_0,a_n), Set)$ is, on objects, the inference function (Def \ref{definition:inference}).
\end{theorem}

\begin{proof}
By Lemma \ref{lemma:coherent_nnl} this is true for all of the single layer neural networks. Any architecture may be formed via the composition of neural network layers. By Lemma \ref{lemma:composition} the corresponding coherent functors of neural network layers may be composed via pullback of their internal spans. For architectures with spans 
\begin{equation}
R^n\xleftarrow[]{Id}R^n\xrightarrow[]{\iota(t)}R^m
\end{equation} 
and
\begin{equation}
R^m\xleftarrow[]{Id}R^m\xrightarrow[]{\iota'(t)}R^k
\end{equation}
Their pullback forms the following span.
\begin{equation}
R^n\xleftarrow[]{Id}R^n\xrightarrow[]{\iota'(t)\iota(t)}R^k
\end{equation}
This is the span the composite of models is sent to under $Coh(\iota'\circ\iota, Set)$ which matches the definition of inference given in \ref{definition:inference}.
\end{proof}

Having demonstrated that a coherent category exists for any dense feed-forward floating point neural network, it is possible to increase the expressivity of this representation by incorporating weight and bias fixing and tying. These constraints, which include sparse and convolutional architectures as special cases, impose equalities between global elements of $R$ in $G$, either between weight entries, between bias entries, or between a weight or bias entry and a fixed value. Each such equality may be enforced by a 2-coequaliser construction in $Coh_\sim$, which identifies the equaliser of the two morphisms with the terminal object, forcing their equality in all $Set$-models.

\begin{theorem}
\label{corollary:weight_tying}
For any floating point neural network architecture $NN(\vec{\sigma}, \vec{a})$ with a finite collection of weight fixing and tying constraints, there exists a coherent category $H$ and coherent functor $\iota' : RSpan(a_0, a_n) \rightarrow H$ such that the $Set$-models of $H$ are exactly the neural networks of $NN(\vec{\sigma}, \vec{a})$ satisfying the constraints, and $Coh(\iota', Set) : Coh(H, Set) \rightarrow Coh(RSpan(a_0, a_n), Set)$ is, on objects, the inference function.
\end{theorem}

\begin{proof}
By Thm \ref{theorem:neural_network_representation}, there exists a coherent category $G$ and coherent functor $\iota : RSpan(a_0, a_n) \rightarrow G$ such that the $Set$-models of $G$ are neural networks of $NN(\vec{\sigma}, \vec{a})$. It suffices to show that each weight or bias fixing or tying constraint may be enforced by a single 2-coequaliser step preserving the inference functor; the full result then follows by applying the construction iteratively over the finite collection of constraints, each step yielding a further coherent functor, with each 2-coequaliser existing by Thm \ref{theorem:bicomplete}.

Consider a single constraint asserting the equality of two global elements. Any weight entry $w_{ij}$ and bias entry $\beta_k$ may be expressed as global elements $\pi_{ij} \circ w : \mathbf{1} \rightarrow R$ and $\pi_k \circ \beta : \mathbf{1} \rightarrow R$ in $G$ respectively, where $\pi_{ij}$ and $\pi_k$ are the appropriate projections of the weight morphism $w : \mathbf{1} \rightarrow (R^n)^m$ and bias morphism $\beta : \mathbf{1} \rightarrow R^m$. Each constraint takes one of the following forms, in each case producing two morphisms $f, g : \mathbf{1} \rightarrow R$ in $G$.
\begin{itemize}[noitemsep]
    \item Weight tying ($w_{ij} = w_{kl}$): $f := \pi_{ij} \circ w$,\quad $g := \pi_{kl} \circ w$.
    \item Bias tying ($\beta_j = \beta_k$): $f := \pi_j \circ \beta$,\quad $g := \pi_k \circ \beta$.
    \item Weight fixing ($w_{ij} = c$): $f := \pi_{ij} \circ w$,\quad $g := c$, where $c : \mathbf{1} \rightarrow R$ is the global element corresponding to the fixed value, which exists in $G$ by Lemma \ref{lemma:float} and Thm \ref{theorem:hard_coding}.
    \item Bias fixing ($\beta_k = c$): $f := \pi_k \circ \beta$,\quad $g := c$, with $g$ as above.
\end{itemize}

Since $G$ has all finite limits, the equaliser $E := eq(f, g) \rightarrowtail \mathbf{1}$ exists in $G$. Both $E$ and $\mathbf{1}$ are objects of $G$ and therefore correspond, by Def \ref{definition:category_sort} and Thm \ref{theorem:schema}, to coherent functors $E, \mathbf{1} : Sort \rightarrow G$. By Thm \ref{theorem:bicomplete}, the 2-coequaliser of these functors exists in $Coh_\sim$. Let $G' := coeq(E, \mathbf{1})$, with coherent functor $e : G \rightarrow G'$ and 2-isomorphism
\[e(E) \cong e(\mathbf{1}) = \mathbf{1}_{G'}.\]

Since $e$ is a coherent functor it preserves equalisers, giving $e(E) \cong eq(e \circ f,\, e \circ g)$ as a subobject of $\mathbf{1}_{G'}$. Combined with the 2-isomorphism above, $eq(e \circ f,\, e \circ g) \cong \mathbf{1}_{G'}$. For any $Set$-model $M : G' \rightarrow Set$, preservation of finite limits by $M$ gives
\[M\!\left(eq(e \circ f,\, e \circ g)\right) = eq(M \circ e \circ f,\, M \circ e \circ g).\]
Since $M(\mathbf{1}_{G'}) = \{\ast\}$, this forces $M \circ e \circ f = M \circ e \circ g$, so every $Set$-model of $G'$ satisfies $f = g$. Conversely, any $Set$-model $M$ of $G$ satisfying $f = g$ has $M(E) = M(eq(f,g)) \cong M(\mathbf{1})$, so the required 2-isomorphism holds and $M$ factors through $G'$ by the universal property of the 2-coequaliser. Therefore $Set$-models of $G'$ are exactly $Set$-models of $G$ satisfying $f = g$.

After iterating the construction over all constraints, let $H$ denote the resulting coherent category and $e : G \rightarrow H$ the composite of all coequaliser maps. The coherent functor $\iota' := e \circ \iota : RSpan(a_0, a_n) \rightarrow H$ is the required interpretation, and $Coh(\iota', Set)$ is the inference functor.
\end{proof}

%% file: sections/conclusion.tex
\section{Conclusion}

This paper has demonstrated an interesting connection between the study of databases, formal theories, and machine learning. As with many areas of category theory, it is common in the category theory for machine learning literature to attempt to construct sensible categories of machine learning models. While the methodology of constructing such categories presented in this work may not generate definitively correct categories for the study of machine learning, it does provide a reasonably generic starting point and a stepping stone for future work. Particularly, as phrasing this construction through the use of formal theories demonstrates a clearer role that machine learning models provide in the storage of information as they learn, and their description as models of formal theories allow at least some of the constraints on their behaviour to be explicitly presented within the syntax of a formal logic.

There remains an open question around the presentation of machine learning algorithms in the language of category theory. However, the suggestion that machine learning models are databases and that model selection may be a Kan extension provides a philosophically satisfying characterisation of the domain. Demonstrating that the inference functor for neural networks comes directly from an extension problem in $Coh$ shows that it may be sensible to think of machine learning as a Kan extension. In particular a description of machine learning using formal theories recontextualises the process of learning itself. The theory of datasets demonstrates that there are assumptions about the data that are known to be true, or are at least assumed to be true. The theory of machine learning models, equipped with a model of datasets, is representing the addition of constraints that the data may satisfy. The theory of machine learning models is a hypothesis, a model of data with additional constraints that is believed to be more explanatory than the basic set of known statements. At a high level, the picture of learning as an extension between theories represents the process of learning as a form of hypothesis testing. Given a collection of known facts, how does the inclusion of additional hypothetical truths affect models of the data? In practice, one may test multiple different architectures and algorithms on different datasets, which from this lens would be the process of testing different hypotheses to determine whose extensions of the data are most faithful to the original observations.

Returning to a more database oriented perspective, previous observations that Kan extensions of databases are data migration operations \cite{spivak_functorial_2013}, such as join and merge, combined with the notion that machine learning may be a Kan extension between data bases representing machine learning models supports the intuitive interpretation that learning is a data migration operation. Unfortunately, one of the limitations of this work is that it does not demonstrate that the 2-morphisms of $Coh$, inherited directly as the natural transforms of $Cat$, are adequate for this interpretation. It is likely that such a naive assignment of 2-morphisms will not produce the desired Kan-extensions and that future work should either investigate alternative 2-morphisms, or possibly alternative logics to generate the desired structure.

The theorems of this work may provide some useful tools for future consideration, particularly as an indirect methodology for constructing formal theories. Thm \ref{theorem:schema} and Thm \ref{theorem:hard_coding} show that objects with known structures, such as categories and functors, may have their data lifted into a representation as a formal theory, allowing one to work with the information without being concerned about the particular sequents of the theory used to present this information. Additionally, the constructions presented in Section \ref{section:mlmodels} show how the construction of theories via colimits allow one to assert various properties of their models through the universal properties of colimits.